\documentclass[14pt]{amsart}
\usepackage{amsfonts}
\usepackage{bm}
\usepackage{amsfonts,amsmath,amssymb,amscd,bbm,amsthm,mathrsfs,dsfont}
\usepackage{mathrsfs}
\usepackage{pb-diagram}
\usepackage{color}
\usepackage{amssymb}
\usepackage{xypic}
\usepackage[dvips]{graphicx}

\usepackage[all]{xy}

\newtheorem{Thm}{Theorem}[section]
\newtheorem{Algo}[Thm]{Algorithm}
\newtheorem{Lem}[Thm]{Lemma}
\newtheorem{Def}[Thm]{Definition}
\newtheorem{Cor}[Thm]{Corollary}
\newtheorem{Prop}[Thm]{Proposition}
\newtheorem{Ex}[Thm]{Example}
\newtheorem{Rem}[Thm]{Remark}

\title{Positive definite matrices in the view-points of planar networks and cluster subalgebras}

\author{Fang Li $\;\;\;\;\;\;$ Yichao Yang}
\address{Fang Li
\newline Department
of Mathematics, Zhejiang University (Yuquan Campus), Hangzhou, Zhejiang
310027, P.R.China}
\email{fangli@zju.edu.cn}

\address{Yichao Yang
\newline Department
of Mathematics, Zhejiang University (Yuquan Campus), Hangzhou, Zhejiang
310027,  P.R.China}
\email{yyc880113@163.com}

\date{version of \today}

\newcommand{\lra}{\longrightarrow}

\newcommand{\ra}{\rightarrow}
\newcommand{\sdp}{\times\kern-.2em\vrule height1.1ex depth-.05ex}
\newcommand{\epi}{\lra \kern-.8em\ra}

 \normalbaselines
\begin{document}

\renewcommand{\thefootnote}{\alph{footnote}}
\setcounter{footnote}{-1} \footnote{\emph{Mathematics Subject
Classification(2010)}: 15B48, 15B57, 13F60, 05E15.}
\renewcommand{\thefootnote}{\alph{footnote}}
\setcounter{footnote}{-1} \footnote{ \emph{Keywords}:  positive definite matrix, generalized Jacobi matrix, planar network, double wiring diagram, cluster subalgebra}

\begin{abstract}
As an improvement of the combinatorial realization of totally positive matrices via the essential positive weightings of certain planar network by
 S.Fomin and A.Zelevisky \cite{[4]}, in this paper, we give the test method of positive definite matrices via the planar networks and the so-called cluster subalgebra respectively, introduced here originally.

This work firstly gives a combinatorial realization of all matrices through planar network, and then sets up a test method for positive definite matrices by $LDU$-decompositions and the horizontal weightings of all lines in their planar networks. On the other hand, mainly the relationship is built between
positive definite matrices and cluster subalgebras.
\end{abstract}

\maketitle
\bigskip

\section{Introduction}
Totally positive matrices and positive definite matrices are two important classes of matrices with certain (partial) positivity. S.Fomin and A.Zelevisky \cite{[4]} gave a combinatorial realization of totally positive matrices via the essential positive weightings of certain planar network and show that in order to test the totally positivity of a matrix, it is enough to check the totally positivity of chamber minors associated to the double wiring system due to  the theory of cluster algebras.

In this paper, we improve this method in \cite{[4]} to matrices with partial positivity.
 We firstly realize all matrices via generalized elementary Jacobi matrices in Proposition \ref{thm2.5} and then test other classes of matrices which share some properties about ``positivity'', e.g., positive definite matrices via the planar networks and the so-called cluster subalgebra respectively, introduced here originally.

 On one hand, it is proven that an invertible Hermitian matrix is positive definite if and only if the weightings of all horizontal lines in its corresponding planar network  are positive in Proposition \ref{prop2.6} and Corollary \ref{maincor}. On the other hand, for any $n\times n$ Hermitian matrix $M$, in Theorem \ref{thm3.3} we prove that one can construct a double wiring system $\Omega$ and a certain subset $S$ of $Q(\Omega)_{0,m}$ such that the matrix $M$ is positive definite if and only if all minors corresponds to any extended cluster of the cluster subalgebra $\mathcal{A'}_{S}$ of $\mathcal{A}_{\Omega}$ on $M$ are positive.
In the last section, we give a summary to explain the relations between the methods of planar networks and cluster subalgebras through double wiring diagrams as a bridge.

This work demonstrates the meaning of the concept of cluster subalgebra for studying some ``partial positivity"  of matrices. In further research, we will be able to use the notion of cluster subalgebras to relate some other important partial positivity of matrices with the theory of cluster algebras and  the combinatorial method.

\section{Preliminaries on some concepts and results}

\subsection{ Totally positive matrices and positive definite matrices.\\}

In linear algebra, matrix theory and their applications, two classes of matrices, that is, totally positive matrices and positive definite matrices,  are very important. Let us recall their definitions.

A square matrix over the complex field $\mathds C$ is said to be \emph{totally positive} if all of its minors are positive real numbers. The first systematic study on this class of matrices was undertaken in the $1930$s by F.R.Gantmacher and M.G.Krein \cite{[5]}, where some remarkable spectral properties were  established, e.g., an $n\times n$ totally positive matrix has $n$ distinct positive eigenvalues. For more details, see \cite{[5]},\cite{[6]},\cite{[7]}.

 A $n\times n$ Hermitian matrix $M$ is said to be \emph{positive definite} if $z^{*}Mz$ is always a positive real number for any $0\not=z\in \mathds C^n$, where  $z^{*}=\bar{z}^\top$. The class of positive definite matrices plays an important role since it is closely related to positive-definite symmetric bilinear forms and inner product vector spaces.

 It is well-known that a Hermitian matrix $M$ is positive definite if and only if all leading principal minors of $M$ are positive. Hence if a totally positive matrix $M$ is Hermitian, then it is obviously positive definite.

\subsection{ The matrices associated with planar networks.\\}

In \cite{[4]}, S.Fomin and A.Zelevinsky gave the combinatoric way to realize all the totally positive matrices via the planar network as follows.

A \emph{planar network} $(\Gamma,\omega)$ is an acyclic directed planar
graph $\Gamma$ whose edges $e$ are assigned
the scalar weights $\omega(e)$. We always assume the edges of $\Gamma$ to be directed from left to right. Assume each network has $n$ sources and $n$
sinks located at the left and right sides in the picture respectively, and
numbered in bottom-to-top.

Now we can associate a planar network $(\Gamma,\omega)$ a matrix $x(\Gamma,\omega)$, which is called the {\em weight matrix}, more precisely, $x(\Gamma,\omega)$ is an $n\times n$ matrix whose $(i,j)$-entry is the sum of weights of all paths from the source $i$ to the sink $j$, where the weight of a (directed) path in $\Gamma$ is defined as the product of weights of all edges in this path. For examples, Figure 1 below includes two planar networks.
\begin{figure}[h] \centering
  \includegraphics*[110,477][540,585]{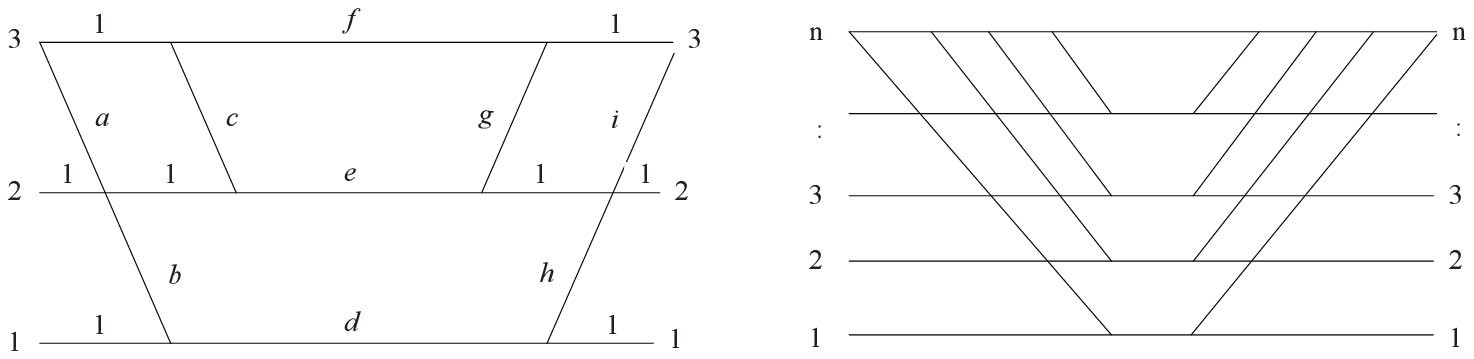}

 Figure 1
\end{figure}

The corresponding weight matrix of the planar network on the left of Figure $1$ is written as below:
\[  \begin{matrix}
  \begin{pmatrix}
   d  &  dh  &  dhi  \\
   bd  &  bdh+e  &  bdhi+eg+ei  \\
   abd  &  abdh+ae+ce  &  abdhi+(a+c)e(g+i)+f
  \end{pmatrix}
  \end{matrix}
\]

In order to calculate minors of the weight matrix, the following lemma is useful, where $\Delta_{I,J}(x)$ denotes the minor of a matrix $x$ with the row set $I$ and the column set $J$ and the weight of a collection of directed paths in $\Gamma$ is defined to be the product of their weights.

\begin{Lem}{\rm (Lemma 1, \cite{[4]})}\label{lem1.1}
{\rm A minor $\Delta_{I,J}(x)$ of the weight matrix of a planar network
is equal to the sum of weights of all collections of vertex-disjoint
paths that connect the sources labeled by $I$ with the sinks labeled
by $J$.}
\end{Lem}

\begin{Rem}\label{rem1.2}
{\rm For the planar network in the above example, or more generally the planar network $(\Gamma,\omega)$ with the form ``$\backslash\backslash -- //$'', the leading principal minors of the weight matrix $x(\Gamma,\omega)$ can be calculated by Lemma \ref{lem1.1} as follow:
\begin{equation}\label{remark1.2}
\Delta_{1,1}=\prod_{i=1}^{t_{1}}\omega_{1i}, \Delta_{12,12}=\prod_{i=1}^{t_{1}}\omega_{1i}\prod_{i=1}^{t_{2}}\omega_{2i},\cdots, \Delta_{12\cdots n,12\cdots n}=\prod_{i=1}^{t_{1}}\omega_{1i}\prod_{i=1}^{t_{2}}\omega_{2i}\cdots \prod_{i=1}^{t_{n}}\omega_{ni}
 \end{equation}
 where the horizontal weightings $\omega_{k1},\omega_{k2},\cdots,\omega_{kt_{k}}$ appear in Line $k$ for $k=1,2,\cdots,n$ respectively. Indeed, the collections of vertex-disjoint paths that connect the sources labeled by $\{$$12\cdots s$$\}$ with the sinks labeled by $\{$$12\cdots s$$\}$ $1\leq s\leq n$ must be horizontal, then the results (\ref{remark1.2}) follows.}
\end{Rem}

In order to state the one-to-one correspondence between the set of totally
positive matrices and that of all positive weightings of a fixed planar
network, we need more notations.

Denote the planar network on the right of Figure $1$ as $\Gamma_{0}$ and call an edge of $\Gamma_{0}$ \emph{essential} if it either is slanted or is one of the $n$ horizontal edges in the middle of $\Gamma_{0}$. Note that $\Gamma_{0}$ has exactly $2\times(1+2+\cdots+(n-1))+n=n^2$ essential edges. Moreover, a weighting $\omega$ of $\Gamma_{0}$ is called \emph{essential} if $\omega(e)\neq 0$ for any essential edge $e$ and $\omega(e)=1$ for all other edges.

The following theorem in \cite{[4]} characterizes all totally positive
$n\times n$ matrices via the essential positive weightings of the planar network.
\begin{Thm}{\rm (Theorem 5, \cite{[4]})}\label{thm1.2}
{\rm The map $\omega \mapsto x(\Gamma_{0},\omega)$ restricts to a bijection between the set of all essential positive weightings of $\Gamma_{0}$ and that of all totally positive
$n\times n$ matrices.}
\end{Thm}

\subsection{ Cluster algebras.\\}

The original motivation of cluster algebras was to give the combinatorial characterizations of the dual canonical basis of the quantized enveloping algebra of a quantum group and of the total positivity for algebraic groups. Now, it plays a prominent role in cluster theory and arises in connection with integrable systems, algebraic Lie theory, Poisson geometry, total positivity, representation theory, combinatorics and etc.

Throughout this paper, we only restrict our attention to cluster algebras given by the so-called \emph{cluster quivers} $Q$, which are directed (multi-)graphs with no loops and no $2$-cycles. Some vertices of $Q$ are denoted as \emph{mutable} and the remaining ones are called \emph{frozen}. All the terminologies are collected in the following definition. For more details, see \cite{[2]}.

\begin{Def}\label{def1.3}
{\rm (i)~ Let $v$ be a mutable vertex of a cluster quiver $Q$. The \emph{quiver mutation} of $Q$ at $v$ is an operation that produces another quiver $Q'=\mu_{v}(Q)$ via the following three steps:

(1)~ Add a new edge $u\rightarrow w$ for each pair of edges $u\rightarrow v$, $v\rightarrow w$ in $Q$, except in the case when both $u,w$ are frozen.

(2)~ Reverse the direction of all edges adjacent to $v$.

(3)~ Remove $2$-cycles until none remains.

(ii)~ Let $\mathcal{F}\supset \mathds{R}$ be any field. We say that a pair $t=(Q,{\bf z})$ is a \emph{seed} in $\mathcal{F}$ if $Q$ is a cluster quiver and {\bf z}, called the \emph{extended cluster}, is a set consisting of some algebraically independent elements of $\mathcal{F}$, as many as the vertices of the quiver $Q$.

(iii)~ The elements of {\bf z} corresponding to the mutable vertices in $Q$ are called \emph{cluster variables} and those corresponding to the frozen vertices in $Q$ are called \emph{frozen variables}.

(iv)~ The \emph{seed mutation} $\mu_z$ at a cluster variable $z$ transforms the seed $t=(Q,{\bf z})$ into a new seed $t'=(Q',{\bf z'})=\mu_{z}(Q,{\bf z})$, where $Q'$ is the quiver resulted after mutating $Q$ at the vertex corresponding to $z$ and ${\bf z'}={\bf z}\cup \{z'\}\backslash \{z\}$, where the new variable $z'$ is subject to the \emph{mutation relation at $z$} via $zz'=\prod \limits_{z\leftarrow x} x+ \prod \limits_{z\rightarrow y} y$. One sometimes calls this relation the \emph{mutation relation at $v$} for the vertex $v$ of $Q$ corresponding to $z$.

(v)~ Two seeds $t=(Q,{\bf z})$ and $t'=(Q',{\bf z'})$ are said to be \emph{mutation-equivalent} if one of them can be obtained from the other after a sequence of seed mutations.

 (vi)~ The \emph{cluster algebra} $\mathcal{A}(Q,{\bf z})$ generated by an \emph{initial seed} $t=(Q,{\bf z})$ is defined as the subring of $\mathcal{F}$ generated by the elements of extended clusters that are mutation-equivalent to $t$.}
\end{Def}

\bigskip
\section{Characterization of positive definite matrices via planar networks.\\}

\subsection{Realization of matrices via generalized elementary Jacobi matrices.\\}

In this section, we would like to realize all $n\times n$ matrices via certain planar networks and then give a way to test whether the matrix is positive definite or not.

For any $n\times n$ matrix $M$, we attempt to find a planar network $(\Gamma,\omega)$ such that its weight matrix $x(\Gamma,\omega)=M$. For this aim, we first consider the case for invertible matrices.

Since any invertible matrix can be written as a product of elementary matrices, let us recall the definition of elementary Jacobi matrices and their connection with the planar networks as in \cite{[4]}.

\begin{Def}\label{def2.1}

{\rm (i)~  Let $E_{i,j}$ be the $n\times n$ matrix whose $(i,j)$-entry is $1$ and all other entries are $0$. For $i=1,2,\cdots,n-1$ and $t\in \mathds{C}$, let $x_{i}(t)=I+tE_{i,i+1}$, $x_{\overline{i}}(t)=I+tE_{i+1,i}$. For $i=1,2,\cdots,n$ and $t\in \mathds{C}\backslash \{0\}$, let $x_{\textcircled{i}}(t)=I+(t-1)E_{i,i}$.

 (ii)~  The \emph{elementary Jacobi matrices} are  matrices $x\in GL_{n}(\mathds{C})$ that differ from the identity matrix $I_n$ in a single entry located either on the main diagonal or immediately above or below it, more precisely, those are ones of the forms  $x_{i}(t),x_{\overline{i}}(t),x_{\textcircled{i}}(t'),~ \forall t\in \mathds{C},t'\in \mathds{C}\backslash \{0\}$.}
\end{Def}

\begin{Lem}{\rm (Theorem 12, \cite{[4]})}\label{lem2.2}
{\rm The three classes of elementary Jacobi matrices are weight matrices of the
following three ``chip''$($Seeing Figure $2$ below$)$, respectively.
In each ``chip'', all edges but one have weight $1$; the distinguished edge
has weight $t$. Slanted edges connect horizontal levels $i$ and $i+1$,
counting from the bottom. Additionally, the concatenation of these ``chip''s corresponds to multiplying their weight matrices.}
\end{Lem}

\begin{figure}[h] \centering
  \includegraphics*[164,363][494,467]{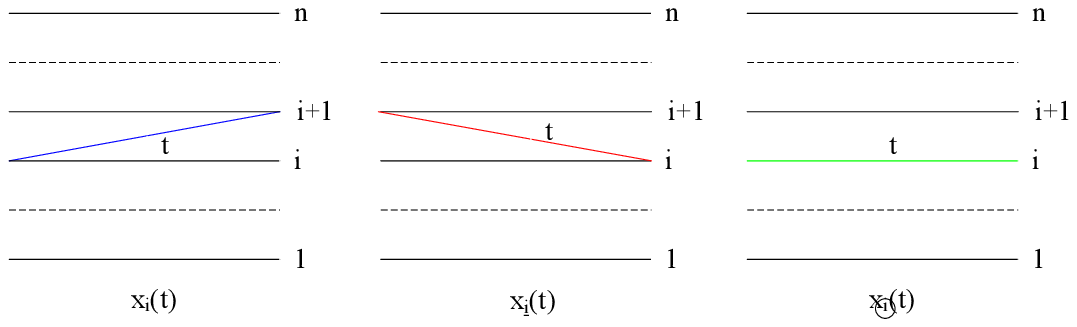}

 Figure 2
\end{figure}

Therefore, any invertible matrix generated by some elementary Jacobi matrices can be realized as the weight matrix for a planar network, since we have the characterization for $GL_{n}(\mathds{C})$.
\begin{Prop}\label{prop2.3}
{\rm For any integer $n\geq 0$, the general linear group $GL_{n}(\mathds{C})$ can be generated by all elementary Jacobi matrices, that is,
$$GL_{n}(\mathds{C})=\langle x_{i}(t),x_{\overline{i}}(t),x_{\textcircled{i}}(t'),t\in \mathds{C},t'\in \mathds{C}\backslash \{0\}\rangle.$$}
\end{Prop}

\begin{proof}
  Clearly,  $GL_{n}(\mathds{C})\supseteq \langle x_{i}(t),x_{\overline{i}}(t),x_{\textcircled{i}}(t'),t\in \mathds{C},t'\in \mathds{C}\backslash \{0\}\rangle$.

  Since any matrix in $GL_{n}$ can be written as a product of elementary matrices,  it suffices to show that any one of elementary matrices of three types  can be generated by some elementary Jacobi matrices.

{\bf The case $1$:~ The type of row multiplication.}

Obviously, all elementary matrices of this type have the form $x_{\textcircled{i}}(t)$ for $t\in \mathds{C}\backslash \{0\}$.

{\bf The case $2$:~ The type of row addition.}

For any $1\leq i,j\leq n$, $i\neq j$, we only consider the case $i<j$, no loss of generality.
We use the induction method on the number $j-i$.

If $j-i=1$, then the result is trivial due to $x_{i}(t)=I+tE_{i,i+1}$.

In general, assume that the conclusion is true for $j-i\leq k$, then considering the case for $j-i=k+1$, we have:
\begin{eqnarray*}
I+tE_{i,i+k+1}  & = &  I+E_{i+k,i+k+1}+tE_{i,i+k+1}-E_{i+k,i+k+1} \\
                & = & (I+tE_{i,i+k})(I+E_{i+k,i+k+1})(I-tE_{i,i+k})(I-E_{i+k,i+k+1})
\end{eqnarray*}
 due to $E_{i,j}E_{k,l}=\delta_{j,k}E_{i,l}$, where $\delta_{j,k}$ is the Kronecker symbol. Using the induction method, it means that any such elementary matrix  can be generated by some elementary Jacobi matrices.

{\bf The case $3$:~ The type of row switching.}

The elementary matrix for exchanging the rows $i$ and $j$ is just equal to the matrix $I+E_{i,j}+E_{j,i}-E_{i,i}-E_{j,j}$, moreover, which can be represented as  \begin{eqnarray*}
I+E_{i,j}+E_{j,i}-E_{i,i}-E_{j,j} & = &  I+E_{j,i}-E_{i,j}-E_{i,i}-E_{j,j}-2E_{j,j}+2E_{i,j}+2E_{j,j}  \\
& = & (I-E_{i,j})(I+E_{j,i})(I-E_{i,j})(I-2E_{j,j}).
\end{eqnarray*}
Then the result follows from that of the case 2.
\end{proof}

Since all elementary Jacobi matrices are invertible, an non-invertible matrix cannot be represented as their product. So, in order to consider not necessarily invertible matrices, we need to generalize the notion of elementary Jacobi matrices.
\begin{Def}\label{def2.4}
{\rm Define a \emph{generalized elementary Jacobi matrix} in $M_n(\mathds C)$ to be a matrix that differs from the identity matrix $I$ in a single entry located either on the main diagonal or immediately above or below it, more precisely, that is one of the form $x_{i}(t),x_{\overline{i}}(t),x_{\textcircled{i}}(t)$ for $t\in \mathds{C}$.}
\end{Def}
\begin{Prop}\label{thm2.5}
{\rm For any positive integer $n$,  any $n\times n$ matrix $M$ over $\mathds C$ can be generated by some generalized elementary Jacobi matrices, that is,   $M_{n}(\mathds{C})=\langle x_{i}(t),x_{\overline{i}}(t),x_{\textcircled{i}}(t),~ \forall t\in \mathds{C}\rangle$.}
\end{Prop}

\begin{proof}
We know that $M=M_{1}\left(
                       \begin{array}{cc}
                         I_r &  \\
                          & O_{n-r} \\
                       \end{array}
                     \right)
M_{2}$, where both $M_{1},M_{2}$ are invertible, $r$ is the rank of $M$. The result follows from Proposition \ref{prop2.3} and $\left(
                       \begin{array}{cc}
                         I_r &  \\
                          & O_{n-r} \\
                       \end{array}
                     \right)=\prod\limits_{s=r+1}^n x_{\textcircled{s}}(0)$.
\end{proof}

By Lemma \ref{lem2.2} and Proposition \ref{thm2.5}, we get that
\begin{Cor}\label{corrplanarnetwork}
{\rm For any matrix $M$, there exists a planar network $(\Gamma,\omega)$ such that its weight matrix $x(\Gamma,\omega)=M$.}
\end{Cor}

\subsection{The criterion for positive definite matrices via planar networks.\\}

We have already seen that by Corollary \ref{corrplanarnetwork}, for any matrix $M$, there exists a planar network $(\Gamma,\omega)$ such that its weight matrix $x(\Gamma,\omega)=M$, although this correspondence is neither injective nor surjective in general. Thus, by Lemma \ref{lem1.1} we can read the leading principal minors $\Delta_{\{1\},\{1\}}, \cdots, \Delta_{[1,n],[1,n]}$ directly from the planar network $(\Gamma,\omega)$. From the basic fact that a Hermitian matrix is positive definite if and only if all of its leading principal minors are positive, we can obtain the algorithm for checking whether a given matrix $M$ is positive definite or not as follows, which give us a view-point to observe positive definite matrices.

\begin{Algo}
{\rm (1)~ Decompose $M=M_{1}M_{2}\cdots M_{t}$ for some generalized elementary Jacobi matrices $M_1, M_2, \cdots, M_t$;

(2)~ By the correspondences between ``chips'' and generalized elementary Jacobi matrices, draw the planar network $(\Gamma,\omega)$ with $x(\Gamma,\omega)=M$;

(3)~  Calculate the leading principal minors directly from $(\Gamma,\omega)$ by the method given in Lemma \ref{lem1.1};

(4)~ Check whether a given matrix $M$ is positive definite or not.}
\end{Algo}

Additionally, by \cite{[9]}, we obtain that for any $n\times n$ matrix $M$, $M$ can be decomposed as $M=U_{1}LU_{2}$ or $M=L_{1}UL_{2}$ with $L,L_{1},L_{2}$ lower triangular matrices and $U,U_{1},U_{2}$ upper triangular matrices. Thus the planar network have the forms ``// $\backslash\backslash$ //'' or ``$\backslash\backslash$ // $\backslash\backslash$'', respectively. Furthermore, we have

\begin{Prop}\label{prop2.6}
{\rm For any $n\times n$ invertible matrix $M$, the following statements are equivalent:

(i) $M$ has a $LDU$-decomposition, i.e., $M=LDU$ with $L$ a unit lower triangular matrix, $U$ a unit upper triangular matrix and $D$ a diagonal matrix.

(ii) The corresponding planar network $(\Gamma,\omega)$ of $M$ is in the form ``$\backslash\backslash -- //$''.

(iii) All the leading principal minors of $M$ are nonsingular.
}
\end{Prop}

\begin{proof}
Firstly, the equivalence of (i) and (iii) can be found in \cite{[8]}.

(i)$\Longrightarrow$(ii): $M=LDU$ a $LDU$-decomposition. According to the proof of Proposition \ref{prop2.3}, we can give the decompositions  $L=L_{1}L_{2}\cdots L_{r}$, $U=U_{1}U_{2}\cdots U_{s}$, $D=D_{1}D_{2}\cdots D_{t}$, where $L_{1},\cdots,L_{r}$ are in the form $x_{\overline{i}}(\gamma)$, $U_{1},\cdots,U_{s}$ are in the form $x_{i}(\gamma)$ and $D_{1},\cdots,D_{t}$ are in the form $x_{\textcircled{i}}(\gamma)$ respectively. Thus the corresponding planar network  $(\Gamma,\omega)$ of $M$ is in the form ``$\backslash\backslash -- //$''.

(ii)$\Longrightarrow$(iii): Since the corresponding planar network $(\Gamma,\omega)$ of $M$ is in the form ``$\backslash\backslash -- //$'' and $M$ is invertible, by Remark \ref{rem1.2} we have $$\Delta_{12\cdots n,12\cdots n}=\prod_{i=1}^{t_{1}}\omega_{1i}\prod_{i=1}^{t_{2}}\omega_{2i}\cdots \prod_{i=1}^{t_{n}}\omega_{ni}={\rm det}(M)\neq 0$$
where the horizontal weightings $\omega_{k1},\omega_{k2},\cdots,\omega_{kt_{k}}$ appear in Line $k$ for $k=1,2,\cdots,n$ respectively. Hence all the leading principal minors
$$\Delta_{1,1}=\prod_{i=1}^{t_{1}}\omega_{1i}, \Delta_{12,12}=\prod_{i=1}^{t_{1}}\omega_{1i}\prod_{i=1}^{t_{2}}\omega_{2i},\cdots, \Delta_{12\cdots n,12\cdots n}=\prod_{i=1}^{t_{1}}\omega_{1i}\prod_{i=1}^{t_{2}}\omega_{2i}\cdots \prod_{i=1}^{t_{n}}\omega_{ni}$$
are nonsingular.
\end{proof}

It is easy to see that for an $n\times n$ matrix $M=LDU$ a $LDU$-decomposition with $D={\rm diag}(d_{1},d_{2},\cdots,d_{n})$, the leading principal minors of $M$ are $\Delta_{12\cdots k,12\cdots k}=d_{1}d_{2}\cdots d_{k}$ for $k=1,2,\cdots,n$. In this case, if furthermore $M$ is invertible, then by Proposition \ref{prop2.6} and Remark \ref{rem1.2}, we have $d_{k}=\prod\limits_{i=1}^{t_{k}}\omega_{ki}$ for all $1\leq k\leq n$. It follows that

\begin{Cor}\label{maincor}
{\rm An invertible Hermitian matrix $M$ is positive definite if and only if the weightings $d_k=\prod\limits_{i=1}^{t_{k}}\omega_{ki}$ of Line $k$ in the corresponding planar network $(\Gamma,\omega)$ of $M$ are positive for all $1\leq k\leq n$. }
\end{Cor}

We end this section with the following example to illustrate the algorithm for checking positive definite matrices.

\begin{Ex}\label{ex2.6}
{\rm Let $M=\begin{matrix}
  \begin{pmatrix}
   1  &  2  &  4  \\
   2  &  6  &  8  \\
   4  &  8  &  18
  \end{pmatrix}
  \end{matrix}$, then firstly $M$ can be decomposed as

\begin{eqnarray*}
  M & = & \begin{matrix}
  \begin{pmatrix}
   1  &  0  &  0  \\
   2  &  1  &  0  \\
   0  &  0  &  1
  \end{pmatrix}
  \end{matrix}
  \begin{matrix}
  \begin{pmatrix}
   1  &  0  &  0  \\
   0  &  1  &  0  \\
   4  &  0  &  1
  \end{pmatrix}
  \end{matrix}
  \begin{matrix}
  \begin{pmatrix}
   1  &  0  &  0  \\
   0  &  2  &  0  \\
   0  &  0  &  2
  \end{pmatrix}
  \end{matrix}
  \begin{matrix}
  \begin{pmatrix}
   1  &  0  &  4  \\
   0  &  1  &  0  \\
   0  &  0  &  1
  \end{pmatrix}
  \end{matrix}
  \begin{matrix}
  \begin{pmatrix}
   1  &  2  &  0  \\
   0  &  1  &  0  \\
   0  &  0  &  1
  \end{pmatrix}
  \end{matrix} \\
   & = & x_{\overline{1}}(2)x_{\overline{2}}(-1)x_{\overline{1}}(-4)x_{\overline{2}}(1)
  x_{\overline{1}}(4)x_{\textcircled{2}}(2)x_{\textcircled{3}}(2)x_{1}(4)x_{2}(1)x_{1}(-4)x_{2}(-1)x_{1}(2)
  \end{eqnarray*}

Next, the corresponding planar network $(\Gamma,\omega)$ is shown as Figure 3, where the weightings $\omega(e)=1$ for all unlabelled edges $e$.

\begin{figure}[h] \centering
  \includegraphics*[137,413][417,496]{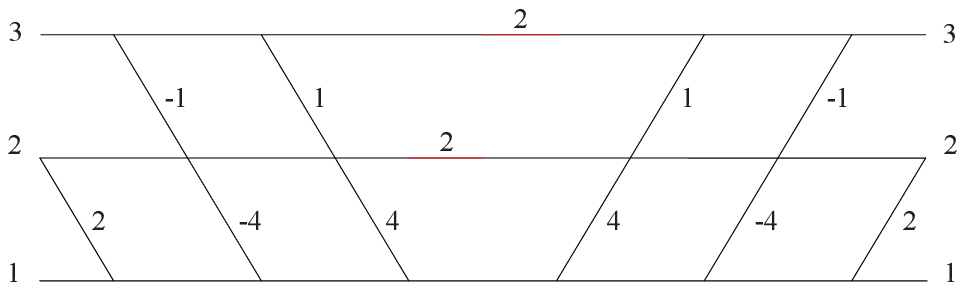}

 Figure 3
\end{figure}

Finally, we have $d_{1}=1, d_{2}=2, d_{3}=2$. Now by Corollary \ref{maincor},  we know that $M$ is positive definite.}
\end{Ex}

\bigskip
\section{Characterization of positive definite matrices via cluster subalgebras.\\}

In \cite{[3]}, S.Fomin show total positivity as one of the motivations of cluster algebras. In brief, take a complex variety $X$ together with a family ${\bf \Delta}$ of ``important'' regular functions $\Delta$ on $X$. The corresponding totally positive variety $X_{>0}$ is the set of points at which all of these functions take positive values, that is, $X_{>0}=\{x\in X: \Delta(x)>0$ for all $\Delta \in {\bf \Delta}\}$. If $X$ is the affine space of matrices of a given size, e.g., if $X=M_{n}(\mathds{C})$, and ${\bf \Delta}$ is the set of all minors, then we recover the classical notation of totally positive matrices.

The question arises naturally that for any element $x\in X$, in order to test whether $x\in X_{>0}$ or not, should we check all the inequalities $\Delta(x)>0$ for all $\Delta \in {\bf \Delta}?$

In the above example, there are $\binom{2n}{n}-1$ minors in total. However, S.Fomin and A.Zelevinsky in \cite{[4]} told us that it is enough to check certain $n^{2}$ minors in order to make sure a matrix to be totally positive through the chamber minors of the corresponding double wiring diagram. This is because all the minors of $M_{n}(\mathds{C})$ can be represented by the rational functions of these $n^{2}$ minors with positive coefficients.

\subsection{Double wiring diagram and its associated quiver.\\}

In order to explain this question, recall that the \emph{double wiring diagram} defined in \cite{[4]} is a diagram consists of two families of $n$ piecewise-straight lines $($each family colored with one of the two colors$)$, the crucial requirement is that each pair of lines with the same color intersect exactly once. Moreover, the lines in a double wiring diagram are numbered separately within each color.

Now we give the definition of \emph{chamber} and \emph{chamber minor} of a double wiring diagram as follows.

\begin{Def}\label{def3.0}
{\rm Let $\Omega$ be a double wiring diagram and $\square$ be a region of $\Omega$. We then assign to $\square$ a pair of subsets $I,J$ of the set $[1,n]=\{1,\cdots,n\}$ such that each subset indicates the line numbers of the corresponding color passing \emph{below} the $\square$. If $|I|=|J|$, then the region $\square$ is called as the \emph{chamber} and in this case, the minor $\Delta_{I,J}$ of an $n\times n$ matrix $x=(x_{ij})$ is called the \emph{chamber minor} of $x$.}
\end{Def}

For example, we show all the chambers of the double wiring diagram in Figure $4$ and the corresponding nine chamber minors $($the total number is always $n^{2}$$)$ are $x_{31}$, $x_{21}$, $x_{11}$, $x_{13}$, $\Delta_{23,12}$, $\Delta_{12,12}$, $\Delta_{12,13}$, $\Delta_{12,23}$ and $\Delta_{123,123}={\rm det}(x)$.

\begin{figure}[h] \centering
  \includegraphics*[73,381][467,529]{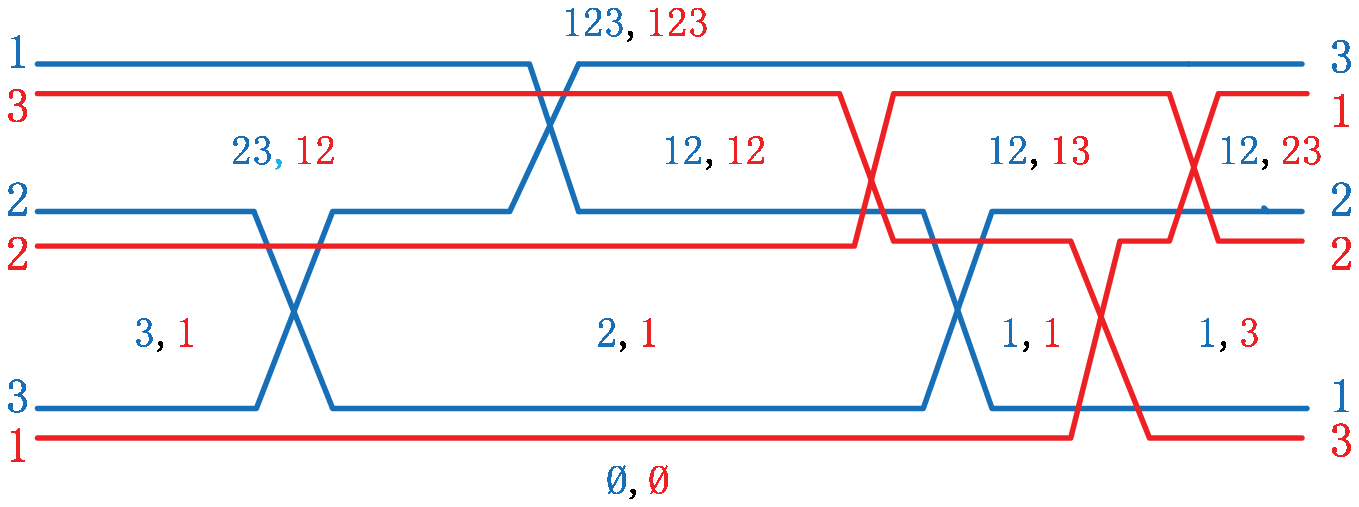}

 Figure 4
\end{figure}

In \cite{[2]}, or more generally in Section $2$ of \cite{[1]}, we can associate a quiver $Q=Q(\Omega)$ to a double wiring diagram $\Omega$ and this construction provided is purely combinatorial.

(i)~ The vertices of $Q$ are the chambers of $\Omega$.

(ii)~ There is an edge between two chambers $c$ and $c'$ of $Q$ in the following cases:

(1)~ They are adjacent chambers in the same row. If the color of the crossing between them is blue, the edge is directed to the left. Otherwise, it is directed to the right.

(2)~ If $c'$ has left and right boundaries of different color and lies completely above $($or below$)$ $c$. If the left boundary of $c'$ is blue, the edge is directed from $c$ to $c'$. Otherwise, it is directed from $c'$ to $c$.

(3)~ If the left boundary of $c'$ is above $c$ and the right boundary of $c$ is below $c'$ and both boundaries have the same color. If such common color is blue, the edge is directed from $c$ to $c'$; otherwise, it is directed from $c'$ to $c$.

(4)~ If the right boundary of $c'$ is above $c$ and the left boundary of $c$ is below $c'$ and both boundaries have the same color. If such common color is blue, the edge is directed from $c'$ to $c$; otherwise, it is directed from $c$ to $c'$.

Then according to (i) and (ii), we get the cluster quiver $Q$ from $\Omega$.

The mutable vertices of $Q(\Omega)$ are the chambers in $\Omega$ are bounded. Then the frozen vertices are the chambers which are unbounded. The following figure on the left shows the quiver associated to the double wiring diagram in Figure $4$, where the vertices surrounded by $\Box$ are frozen and the others are mutable.

\begin{figure}[h] \centering
  \includegraphics*[109,428][514,547]{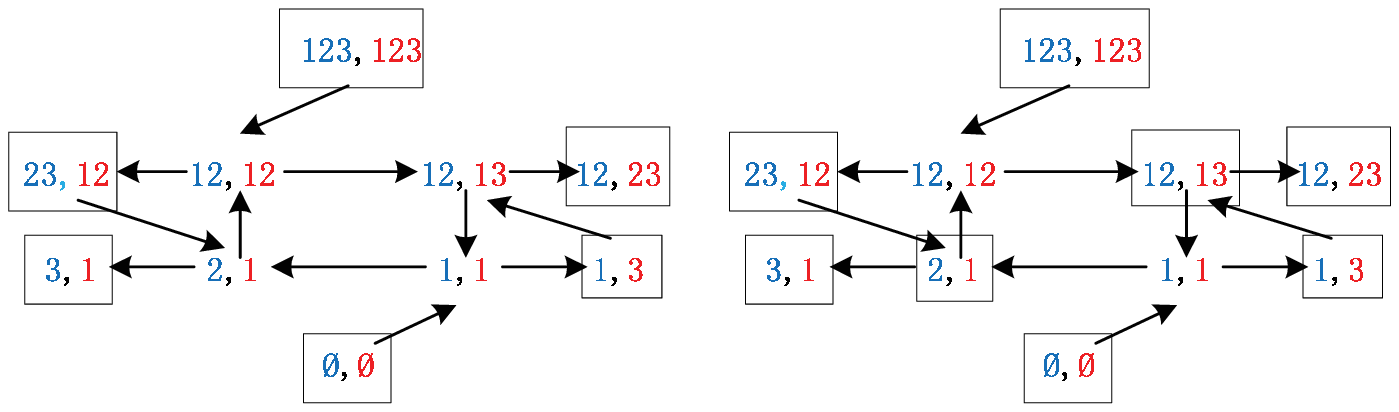}

 Figure 5
\end{figure}

In \cite{[1]}, A.Berenstein, S.Fomin and A.Zelevinsky showed that all the $\binom{2n}{n}-1$ minors of a matrix $M$ occurs as cluster variables in the cluster algebra $\mathcal{A}(\Omega)$ generated by the initial seed $(Q(\Omega),{\bf z}(\Omega))$ associated to a double wiring diagram $\Omega$, where ${\bf z}(\Omega)$ is the collection of all chamber minors in $\Omega$ and thus it is enough to check the positivity of these $n^2$ chamber minors in ${\bf z}(\Omega)$ to make sure of totally positivity of $M$.

\subsection{Cluster subalgebras and positive definite matrices.\\}

In what follows we replace the ``totally positive matrices'' by other classes of matrices which share some properties about ``partial positivity'', such as positive definite matrix, $P$-matrix and so on.

In order to give a similar result for a positive definite matrix, we first introduce some notations. Let $\mathcal{A}=\mathcal{A}(Q,{\bf z})$ be the cluster algebra generated by an initial seed $t=(Q,{\bf z})$ which is defined in Definition \ref{def1.3}. Let $Q_{0}=Q_{0,m}\cup Q_{0,f}$ and ${\bf z}={\bf z}_{m}\cup{\bf z}_{f}$, where $Q_{0},Q_{0,m},Q_{0,f}$ represents the vertices, mutable vertices, frozen vertices of $Q$ respectively and ${\bf z}_{m}, {\bf z}_{f}$ corresponds to the mutable part, frozen part of ${\bf z}$ respectively. Now we raise the definition of the so-called \emph{cluster subalgebra}.

\begin{Def}\label{def3.1}
{\rm Let us keep our foregoing notations. For a subset $S\subseteq Q_{0,m}$, we define the corresponding \emph{cluster subalgebra} associated to $S$ as the cluster algebra $\mathcal{A'}_{S}$ generated by the initial seed $t'=(Q',{\bf z'})$, where $Q'=Q, Q'_{0,m}=Q_{0,m}\backslash S, Q'_{0,f}=Q_{0,f}\cup S$ and $z'_{m}=z_{m}\backslash z_{S},z'_{f}=z_{f}\cup z_{S}$.}
\end{Def}

\begin{Ex}\label{ex3.2}
{\rm Let $\mathcal{A}=\mathcal{A}(\Omega)$, where $\Omega$ is the double wiring diagram on the left of Figure $5$. Then by choosing $S=\{12,13; ~2,1\}$, the corresponding cluster subalgebra $\mathcal{A'}_{S}=\mathcal{A}(\Omega')$, where $\Omega'$ is the double wiring diagram on the right of Figure $5$.}
\end{Ex}

From this example on the $3\times 3$ Hermitian matrix $M$, it is easy to see that in order to test whether $M$ is positive definite or not, it suffices to check only the positivity of any extended cluster from the cluster subalgebra $\mathcal{A}(\Omega')$.

Moreover, for any $n\times n$ Hermitian matrix, we can obtain the following similar conclusion.

\begin{Thm}\label{thm3.3}
{\rm For any $n\times n$ Hermitian matrix $M$, one can construct a double wiring system $\Omega$ and a certain subset $S$ of $Q(\Omega)_{0,m}$ such that the matrix $M$ is positive definite if and only if all minors correspondig to any extended cluster of the cluster subalgebra $\mathcal{A'}_{S}$ of $\mathcal{A}_{\Omega}$ on $M$ are positive.}
\end{Thm}

\begin{proof}
Firstly, we construct the double wiring diagram $\Omega$ as follows: For each row
between two wires, all the blue crossings are on the left side of the red crossings.
Then there are $n-1$ such rows and for each row, there exists an unique chamber with the blue left boundary and the red right boundary. Thus we get $n-1$ such chambers, which correspond by one-to-one to the $n-1$ minors $\Delta_{1,1},\Delta_{12,12},\cdots,\Delta_{12\cdots (n-1), 12\cdots (n-1)}$, that is, $\{1,1;~ 12,12;~ \cdots;~ 12\cdots (n-1), 12\cdots (n-1)\}=Q(\Omega)_{0,m}\backslash S$. Now we choose $S$ be $S=Q(\Omega)_{0,m}\backslash \{1,1;~ 12,12;~ \cdots;~ 12\cdots (n-1), 12\cdots (n-1)\}$.

Finally, the conclusion follows by the fact that $\{12\cdots n, 12\cdots n\}\in Q(\Omega)_{0,f}$ and $M$ is positive definite if and only if all of its leading principal minors are all positive and all its leading principal minors can be expressed as the rational functions of minors corresponds to the extended clusters of $\mathcal{A'}_{S}$ with positive coefficients.
\end{proof}

\bigskip
\section{Further Remarks as Summary.\\}

For any $n\times n$ matrix $M$, on one hand, all its $\binom{2n}{n}-1$ minors share the cluster algebra structure via a double wiring diagram $\Omega$ and its associated quiver $Q(\Omega)$. Thus in order to check some ``positivity'' properties about the minors of $M$, it suffices to check the ``positivity'' of some minors which appears in any clusters in the corresponding cluster subalgebra. On the other hand, for the matrix $M$, there exists a planar network $($$\Gamma$,$\omega$$)$ such that its weight matrix $x(\Gamma,\omega)=M$ and we can calculate any minor $\Delta_{I,J}$ of $M$ for $I,J\subseteq [1,n]$ by Lemma \ref{lem1.1}. Therefore, we conclude this paper by the following Figure.

\begin{figure}[h] \centering
  \includegraphics*[154,482][532,582]{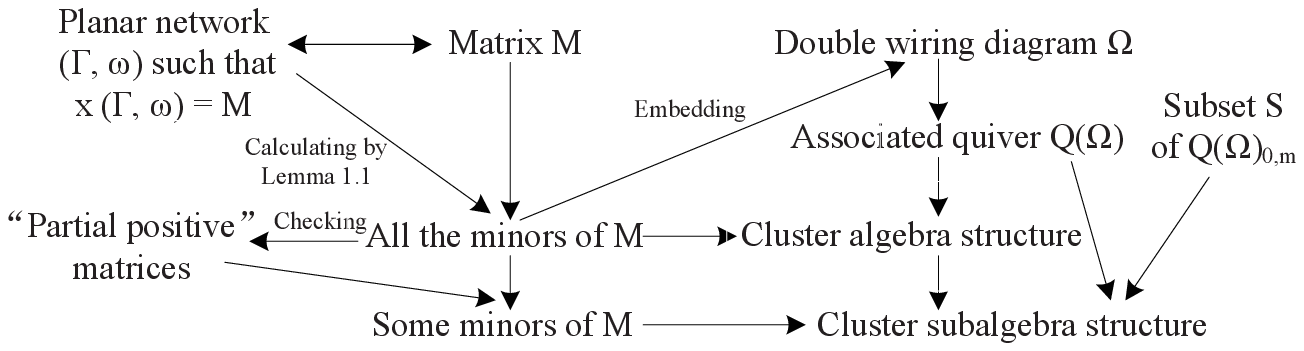}

 Figure 6
\end{figure}

\vspace{5mm}

{\bf Acknowledgements:}\; This project is supported by the National Natural Science
Foundation of
China (No.11271318, No.11171296 and No.J1210038) and the Specialized Research Fund for the
Doctoral Program of Higher Education of China (No.20110101110010) and the Zhejiang
Provincial
Natural Science Foundation of China (No.LZ13A010001).


\begin{thebibliography}{abcdsfgh}

\bibitem{[1]} A.Berenstein, S.Fomin and A.Zelevinsky, Cluster Algebras III: Upper Bounds and Double Bruhat Cells, Duke Math J. $\mathbf{Vol.126}$, $\mathbf{Number 1}$ (2005), 1-52.

\bibitem{[2]} C.Cuenca, http://www.math.umn.edu/~reiner/REU/Cuenca2013.pdf.

\bibitem{[3]} S.Fomin, Total positivity and cluster algebras. (English summary) Proceedings of the International Congress of Mathematicians. Volume II, 125-145, Hindustan Book Agency, New Delhi, 2010.

\bibitem{[4]} S.Fomin, A.Zelevinsky, Total Positivity: Tests
and Parametrizations. Springer-Verlag New York. Volume 22, Number 1,
2000.

\bibitem{[5]} F.R.Gantmacher, M.G.Krein, "Sur les matrices osciilatoires,"
 C.R.Acad.Sci.(Paris) $\mathbf{201}$ (1935), 577-579.

\bibitem{[6]} F.R.Gantmacher, M.G.Krein, Sur les matrices completement non n$\acute{e}$gatives et oscillatoires, Compos.Math. $\mathbf{4}$ (1937), 445-476.

\bibitem{[7]} F.R.Gantmacher, M.G.Krein, Oszillationsmatrizen, Osziltationskerne und Kteine Schwingungen Mechaniseher Systeme AkademieVertag, Bedin, 1960. (Russian edition: Moscow-Leningrad, 1950.)

\bibitem{[8]} R.A.Horn, C.R.Johnson, Matrix analysis. Cambridge University Press, Cambridge, 1985.

\bibitem{[9]} P.Okunev, C.R.Johnson, Necessary and Sufficient Conditions For Existence of the LU Factorization of an Arbitrary Matrix, 1997.

\end{thebibliography}
\end{document}